\documentclass[10pt,a4paper]{article}
\usepackage[latin1]{inputenc}
\usepackage{listings}
\usepackage{amsmath, amsthm, amssymb}
\usepackage{amssymb}
\usepackage{enumerate}
\usepackage{tikz}
\newtheorem{theorem}[equation]{Theorem}
\newtheorem{lemma}[equation]{Lemma}

\newtheorem{proposition}[equation]{Proposition}
\newtheorem{definition}[equation]{Definition}

\newtheorem{remark}[equation]{Remark}
\title{On Neck Singularities for 2-Convex Mean Curvature Flow}

\begin{document}
\author{Alexander Majchrowski} 
\date{}
\maketitle
\numberwithin{equation}{section}
\newcommand{\map}[2]{\,{:}\,#1\!\longrightarrow\!#2}
\newcommand{\defn}[1]{\textbf{#1}}
\newcommand{\newln}{\\&\quad\quad{}}
\renewcommand{\thefootnote}{\fnsymbol{footnote}}

	\section{Introduction}
    Let $F_0:\mathcal{M}\to\mathbb{R}^{n+1}$ be a smooth immersion of an oriented $n$-dimesnional hypersurface in Euclidean space with $n\geq 3$. The evolution of $\mathcal{M}_0=F_0(\mathcal{M})$ by mean curvature flow is the one-parameter family of smooth immersions $F:\mathcal{M}\times[0,T)$, $T<\infty$ satisfying
    \begin{align*}
    \frac{\partial F}{\partial t}(p,t)&=-H(p,t)\nu(p,t),\quad p\in\mathcal{M},t\geq0\\
    &F(\cdot,0)=F_0,
    \end{align*} 
    where $H(p,t)$ and $\nu(p,t)$ are the mean curvature and the outer normal respectively at the point $F(p,t)$ of the surface $\mathcal{M}_t=F(\cdot,t)(\mathcal{M})$. The signs are chosen such that $-H\nu=\vec{H}$ is the mean curvature vector and the mean curvature of a convex surface is positive. 
    We define a surface to be two-convex if the sum of the two smallest eigenvalues is always positive, i.e. $\lambda_1+\lambda_2\geq 0$ everywhere on $\mathcal{M}_0$. For more details regarding Mean curvature flow for convex and 2-convex hypersurfaces please refer to \cite{huisken1984flow}, \cite{huisken1999convexity} and \cite{huisken2009mean}.
    
    In this paper we are dealing with mean curvature flow with surgeries of two-convex hypersurfaces \cite{huisken2009mean}. The main focus is to expand on the discussion in Section $3$ of \cite{huisken2009mean}. Firstly we wish to establish how the neck detection lemma allows us to detect necks where the cross sections will be diffeomorphic to $S^{n-1}$. We then want to see how we are able to glue these cross sections together with full control on their parametrisation - for this we will show we can use a harmonic spherical parametrisation \cite{hamilton1997four}. We then introduce the notion of a normal and maximal necks, this allows us to obtain uniqueness, existence and overlapping properties for normal parametrisations on $(\epsilon,k)$-cylindrical hypersurface necks. Lastly given a neck $N:S^{n-1}\times[a,b]\to\mathcal{M}$ we want to see that in the case that either $a=\infty$ or $b=\infty$ that this forces them to both to be $\infty$ and that we are left with a solid tube $S^{n-1}\times S^1$.

	Acknowledgements:
	I would like to thank Gerhard Huisken for his invaluable help in providing me with the key details for this argument.
	
	I would also like to thank Stephan Tillman for help regarding the topological arguments presented here.   
	
	Last but not least I would like to thank my supervisor Zhou Zhang and the School of Mathematics at The University of Sydney for their help and support.

	\section{Properties of Necks}
	
		We want to begin by showing that an immersed, compact, 2-convex hypersurface undergoing mean curvature flow will develop necks in the regions with large curvature as the singular time is approached. But first we must provide the definition of a curvature and a geometric neck. It is easier to detect curvature necks, using estimates of quantities satisfied by the solutions of the flow. However surgery is only possible on regions diffeomorphic to a cylinder, a geometric neck. Hamilton showed that these two are basically equivalent, \cite{hamilton1997four}.
	
	\begin{definition}[Extrinsic curvature necks]
		Let $\mathcal{M}^n\to\mathbb{R}^{n+1}$ be a smooth hypersurface and $p\in\mathcal{M}^n$.
		\begin{enumerate}[(i)]
			\item We say the extrinsic curvature is $\epsilon$-cylindrical at $p$ is there exists an orthonormal frame at $p$ such that
			\begin{align}
			|W(p)-\bar{W}(p)|\leq\epsilon \label{epneck}
			\end{align}
			where $\bar{W}(p)$ is the Weingarten map on the tangent space to $\mathbb{S}^{n-1}\times\mathbb{R}\to\mathbb{R}^{n+1}$ in a standard frame.
			\item We say the extrinsic curvature is $(\epsilon,k)$-parallel at $p$ if
			\begin{align}
			|\nabla^lW(p)|\leq \epsilon\quad\text{for}\;1\leq l\leq k.\label{epkpar}
			\end{align}
			
			\item We say that $p$ lies at the centre of an $(\epsilon,k,L)$ extrinsic curvature neck if it is $(\epsilon,k)$-parallel $\in B_L(p)$ and the extrinsic curvature is $(\epsilon,k,L)$-hypothetically cylindrical around $p$.
		\end{enumerate}
	\end{definition}
	
	\begin{definition}[Geometric Neck]
		The local diffeomorphism $N:\mathbb{S}^{n-1}\times[a,b]\to(\mathcal{M},g)$ is called an (intrinsic) $(\epsilon,k)$-cylindrical geometric neck if it satisfies the following conditions:
		\begin{enumerate}[(i)]
			\item The conformal metric $\hat{g}=r^{-2}(z)g$ satisfies the estimates
			\begin{align}
			|\hat{g}-\bar{g}|_{\bar{g}}\leq \epsilon,\quad |\bar{D}^j\hat{g}|_{\bar{g}}\leq\epsilon\quad\text{for}\;1\leq j\leq k \label{cylinintrin1}
			\end{align}
			uniformly on $\mathbb{S}^{n-1}\times[a,b]$.
			\item The mean radius function $r:[a,b]\to\mathbb{R}$ satisfies the estimate
			\begin{align}
			|(\frac{d}{dz})^j\log r(z)|\leq\epsilon \label{clinintrin2}
			\end{align}
			for all $1\leq j\leq k$ everywhere on $[a,b]$.
		\end{enumerate}
		Moreover we can say that $N$ is an $(\epsilon,k)$-cylindrical hypersurface neck if in addition to the above assumptions we also have:
		\begin{align}
		|W(q)-r(z)^{-1}\bar{W}|&\leq\epsilon r(z)^{-1}\quad\text{and}\\
		|\nabla^lW(q)|&\leq\epsilon r(z)^{-l-1},\quad 1\leq l\leq k,\label{cylinintrin2}
		\end{align}
		for all $q\in\mathbb{S}^{n-1}\times{z}$ and all $z\in[a,b]$.
	\end{definition}
	
	\begin{definition}
		Given $t,\theta$ such that $0\leq t-\theta<t\leq T_0$, we define the backward parabolic neighbourhood of $(p,t)$ by,
		\begin{align}
		\mathcal{P}(p,t,r,\theta)=\{(q,s)|q\in\mathcal{B}_{g(t)}(p,r),s\in[t-\theta,t]\}. \label{backwardpara}
		\end{align}
		where $\mathcal{B}_{g(t)}(p,r)\subset\mathcal{M}$ is the closed ball of radius $r$ w.r.t. the metric $g(t)$.
	\end{definition}
	
	We define the following to simplify the analysis of necks. 
	\begin{align}
	\hat{r}(p,t):=\frac{n-1}{H(p,t)},\;\hat{\mathcal{P}}:=(p,t,l,\theta):=\mathcal{P}(p,t,\hat{r}(p,t),\hat{r}(p,t)^2\theta). \label{paraneck}
	\end{align}

	The following lemma provides the first step in detecting necks when the curvature is large enough as a singular time is approached.
	
	\begin{lemma}
		Let $\mathcal{M}_t$, $t\in[0,T)$ be a mean curvature flow with surgeries as defined in \cite{huisken2009mean}. Starting from an initial manifold $\mathcal{M}_t\in C(R,\alpha)$ for some $R,\alpha$. Let $\epsilon,
		\theta,L>0$ and $k\geq k_0\geq2$ be given. Then we can find $\eta_0, H_0$ with the following property. Suppose that $p_0\in\mathcal{M}$ and $t\in[0,T)$ are such that
		\begin{enumerate}[(ND1)]
			\item $H(p_0,t_0)\geq H_0$, $\frac{\lambda_1(p_0,t_0)}{H(p_0,t_0)}\leq\eta_0$
			\item The neighbourhood $\hat{\mathcal{P}}(p_0,t_0,L,\theta)$ does not contain surgeries.
		\end{enumerate}
		Then
		\begin{enumerate}[(i)]
			\item The neighbourhood $\hat{\mathcal{P}}(p_0,t_0,L\theta)$ is an $(\epsilon,k_0-1,L,\theta)$-shrinking curvature neck;
			\item The neighbourhood $\hat{\mathcal{P}}(p_0,t_0,L-1,\theta/2)$ us ab $(\epsilon,k,L-1,\theta/2)$ shrinking curvature neck.
		\end{enumerate}
		The constant $\eta_0(\alpha,\epsilon,k,L,\theta)$, whilst $H_0=h_0R^{-1}$ , where $h_0(\alpha,\epsilon,k,L,\theta)$.
	\end{lemma}
	
	We can combine the above lemma with the following proposition found in \cite{hamilton1997four} C3.2, to find that there is a closed cross section with tightly pinched Riemannian curvature. This tells you that there is some diffeomorphism of this cross section to that of a standard sphere $S^{n-1}$, \cite{huisken1985ricci}.
	
	\begin{proposition}\label{graph}
		Let $k\geq 1$. For all $L\geq 10$ there exists $\epsilon(n,L)>0$ and $c(n,L)$ such that at any point $p\in\mathcal{M}$ which lies at the centre of an $(\epsilon,k,L)$ extrinsic curvature neck with $0<\epsilon\leq \epsilon(n,L)$ has a neighbourhood which after appropriate rescaling can be written as a cylindrical f a function $u:S^{n-1}\times[-(L-1),(L-1)]\to\mathbb{R}$ over some standard cylinder in $\mathbb{R}^{n+1}$, satisfying
		\begin{align*}
		||u||_{C^{k+2}}\leq c(n,L)\epsilon
		\end{align*}
	\end{proposition}
	
	The proof of the above can be found in \cite{huisken2009mean} Proposition 3.5.

	Once we know these cross sections  are $(\epsilon,k)$ spherical by Proposition \ref{graph}, we can obtain a harmonic spherical parametrisation, Theorem C1.1 in \cite{hamilton1997four}.
	
	\begin{definition}
		A harmonic spherical parametrisation is of the form $P^*=PF$ where we want
		\begin{align*}
		F(S^n,\bar{g})\to(S^n,g)
		\end{align*}
		to be harmonic from the standard metric $\bar{g}$ to the pull-back metric $g$.
	\end{definition}
	
	\begin{theorem}
		If there exists a geometrically $(\epsilon,k)$ spherical parametrization of $\mathcal{M}$, then there also exists a harmonic spherical parametrization. If $n\geq 3$ it is unique up to rotation.
	\end{theorem}
	
	\begin{remark}
		For $n=2$ it is unique up to a conformal transformation, and hence unique up to a rotation if we also require that the centre of mass of the pull-back metric $g$ on $S^n\subset \mathbb{R}^{n+1}$ lies at the origin $0$. This makes the $n=2$ case more complicated to deal with.
	\end{remark}
	
	This theorem improves on our parametrisation by giving us a harmonic one. This makes the parametrisation rigid and close to the standard parametrisation of the sphere in angular directions, the only freedom left now is the rigid rotation of the standard $S^{n-1}$ in each cross section of the neck. That is, the $z$ coordinate does not matter, we will have the same rotation.
	
	To obtain a unique $z$-coordinate along the neck, we can use the implicit function theorem to make the cross sections of constant mean curvature and then label them by the volume between them, this is shown in the proof of the next Lemma. Since this is an elliptic equation we can get our cross sections even closer to the standard round sphere in higher norms than the first cross sections we found at the beginning. To do so we first need to define a normal neck.

	\begin{definition}
		A topological neck $N$ in a manfiold $\mathcal{M}$ is a local diffeomorphism of a cylinder into $\mathcal{M}$
		\begin{align*}
		N: S^{n-1}\times[a,b]\to(\mathcal{M},g)
		\end{align*}
		The neck is called normal if it satisfies the following conditions:
		\begin{enumerate}[(i)]
			\item Each cross section $\Sigma_z=N(S^{n-1}\times\{z\})\subset(\mathcal{M},g)$ has constant mean curvature.
			\item The restriction of $N$ to each $S^{n-1}\times\{z\}$ equipped with the standard metric is a harmonic map to $\Sigma_z$ equipped with the metric induced by $g$, and
			\item in case $n=3$ only, the centre of mass of the pull-back of $g$ on $S^2\times\{z\}$ considered as a subset of $\mathbb{R}^3\times\{z\}$ lies at the origin ${0}\times\{z\}$.
			\item The volume of any subcylinder with respect to the pullback of g is given by
			\begin{align*}
			Vol(S^{n-1}\times[v,w],g)=\sigma_{n-1}\int_{v}^{w}r(z)^ndz.
			\end{align*}
			\item For any Killing vector field $\bar{V}$ on $S^{n-1}\times\{z\}$ we have that
			\begin{align*}
			\int_{S^{n-1\times\{z\}}} \bar{g}(\bar{V},U)d\mu=0
			\end{align*}
			where $U$ is the unit normal vector field to $\Sigma_z$ in $(\mathcal{M},g)$ and $d\mu$ is the measure of the metric $\bar{g}$ on the standard cylinder.
		\end{enumerate}
	\end{definition}

	The following lemma and proof from \cite{hamilton1997four} C2.1 tells us how to fit all the cross sections together with complete control on their parametrisation.
	
	\begin{lemma}\label{lemmaisom}
		There exists $(\epsilon,k)$ so that if $N_1$ and $N_2$ are necks in the same manifold $\mathcal{M}$ and both are normal and geometrically $(\epsilon,k)$ cylindrical, and if there exists a diffeomorphism $F$ of the cylinders such that $N_2=FN_1$, then $F$ is an isometry in the standard metrics on the cylinders.
	\end{lemma}
	
	\begin{proof}
		For any smooth constant mean curvature hypersurface, there exists a unique one-parameter family of nearby constant mean curvature hypersurfaces by the implicit function theorem. The map F takes an end of one cylinder to an end of the other. Since these constant mean curvature hypersurfaces agree under $F$, so do all the nearly ones; and we can pursue this all the way from one end to the other. Referring to the definition above condition (i) guarantees that $F$ preserves the foliation by horizontal spheres. Given the foliation, condition (ii) together with the geometric closeness to the standard metric makes $F$ act by isometry on each horizontal sphere $S^{n-1}\times\{z\}$. Condition (iv) forces the vertical height functions $z$ to differ by an isometry of $\mathbb{R}$. Lastly condition (v) ensures that the possible rotations in the harmonic spherical parametrisation of each individual cross section are glued together in such a way that there is only one rotation of the standard $S^{n-1}$ left to choose for the whole neck; because by parts (i),(ii),(iv) we are dealing with a map of the cylinder to itself which preserves the height and acts on each horizontal sphere by rotation, and if it is perpendicular to the rotations it must be constant.
	\end{proof}

    It is this rigidity of the parametrisation along the neck that ensures that you are not just somehow diffeomorphic to $S^{n-1}\times[a,b]$ in the neck, but also extremely close (up to rescaling) to the standard metric and parametrisation of the cylinder. In particular this ensures that there is a diffeomorphism unique up to a rotation and close to an isometry between the two cross sections at the ends of a neck.
    
    We now have uniqueness. For existence of normal necks refer to Theorem C2.2 in \cite{hamilton1997four}.
    
    We wish to combine normal necks which are cylindrical enough and overlap more than a little bit near the ends into a single neck.
    Unfortunately Lemma \ref{lemmaisom} is not enough. It tells us that if a diffeomorphism exists then we have isometry, but it does not guarantee the existence of this diffeomorphism $F$. The next theorem and proof from \cite{hamilton1997four} C2.4 will guarantee the existence of such a diffeomorphism and give us the overlapping properties we require.
    
    \begin{theorem}
    	For and $\delta>0$ we can choose $\epsilon>0$ and $k$ with the following property. If $N_1,N_2$ are two normal necks in the same manifold $M$ which are both geometrically $(\epsilon,k)$ cylindrical, and if there is any point $P_1$ in the domain cylinder of $N_1$ at standard distance at least $\delta$ from the ends whose imagine in $M$ is also in the image of $N_2$, then there exists a normal neck $N$ which is also geometrically $(\epsilon,k)$ cylindrical, and there exist diffeomorphisms $F_1$ and $F_2$ such that $N_1=NF_1$ and $N_2=NF_2$, provided $n\geq 3$
    \end{theorem}
	
	\begin{proof}
		If $n\geq 3$ then the cylinder $S^{n-1}\times[a,b]$ is simply connected. Let $P_2\in S^{n-1}\times\{z_2\}$ be a point in the cylinder $N_2$ whose image $P=N_2P_2$ in $\mathcal{M}$ is the same as the image $P=N_1P_1$ of the given $P_1\in S	^{n-1}\times\{z_1\}$. We claim that we can find a map 
		\begin{align*}
		G:S^{n-1}\times\{z_2\}\to S^{n-1}\times\{z_1\}
		\end{align*}
		such that $N_1G=N_2$ and $GP_2=P_1$. To see this we take any path $\gamma_2$ from $P_2$ to any point $Q_2\in S^{n-1}\times\{z_2\}$. Let $\gamma=N_2\gamma_2$ be its projection in $\mathcal{M}$, we then lift $\gamma$ to a path $\gamma_1$ in the first cylinder with $\gamma=N_1\gamma_1$. The point $P_1$ is well in the interior, so we can lift this path until we reach a point $Q_1$ with $N_1Q_1=Q=N_2Q_2$.
		
		
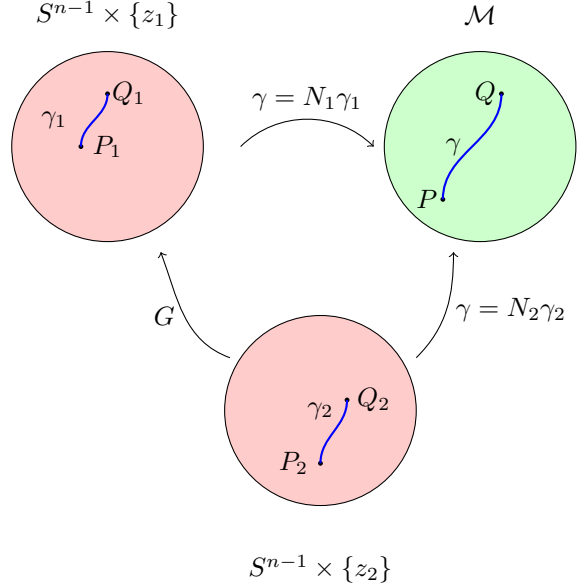
\begin{center}
	
		\begin{tikzpicture}[scale=0.7]

		\filldraw[fill=green!20!white](7.5,0)circle(1.8cm);
		\draw(7.5,2.5)node{$\mathcal{M}$};
		\filldraw(0.3,0.62)circle(1pt);
		\filldraw[fill=red!20!white](0.5,0)circle(1.8cm);
		\filldraw[fill=red!20!white](4.5,-5)circle(1.8cm);
		\draw(0.5,0)node{$P_1$};
		\draw(0.9,1)node{$Q_1$};
		\filldraw(0,0)circle(1pt);
		\filldraw(0.5,1)circle(1pt);
		\draw(6.5,-1)node{$P$};
		\filldraw(6.8,-1)circle(1pt);
		\draw(7.6,1)node{$Q$};
		\filldraw(7.9,1)circle(1pt);
		\draw[->](3,0)to[out=45,in=135]node[above]{$\gamma=N_1\gamma_1$}(5.5,0);
		\draw(0.5,2.5)node{$S^{n-1}\times\{z_1\}$};
		\draw[thick,color=blue](0,0) to[out=90,in=-90] (0.5,1);
		\draw[thick,color=blue](6.8,-1) to[out=90,in=-90] (7.9,1);
		\draw(7,0)node{$\gamma$};
		\draw(-.5,0.5)node{$\gamma_1$};
		\draw(4,-6)node{$P_2$};
		\draw(5.5,-4.8)node{$Q_2$};
		\filldraw(4.5,-6)circle(1pt);
		\filldraw(5,-4.8)circle(1pt);
		\draw[thick,color=blue](4.5,-6) to[out=90,in=-90] (5,-4.8);
		\filldraw(4.5,-5)node{$\gamma_2$};
		\draw[->](6.3,-4)to[out=45,in=-90]node[right]{$\gamma=N_2\gamma_2$}(7,-2);
		\draw(4.5,-8)node{$S^{n-1}\times\{z_2\}$};
		\draw[->](2.8,-4)to[out=160,in=-70]node[left]{$G$}(1.5,-2);
		
		\end{tikzpicture}
	\end{center}
		
		The only case where this would fail would be if $\gamma_1$ ran into the boundary of the first cylinder. But we claim this won't happen as $\gamma_1$ is nearly horizontal. The metric $(\mathcal{M},g)$ will pull back onto metrics $(N_1,g_1)$ and $(N_2,g_2)$, both of these are close to the standard metrics $\bar{g}_1$ and $\bar{g}_2$ on the two cylinders. The horizontal spheres on the standard cylinders are where the Ricci curvatures of the product metric are all $n-1$, while in the vertical direction they are $0$. For $k\geq 0$ the curvatures of $g_1$ are close to $\bar{g}_1$ and $g_2$ are close to those of $\bar{g}_2$. The Ricci curvature in the direction of $\gamma_2$ is close to $n-1$ since it is in $S^{n-1}$, and the Ricci curvature of $g_1$ in the direction of $\gamma_1$ is equal to that of $g_2$ in $\gamma_2$. Therefore $\gamma_1$ is close to horizontal. As long as the path $\gamma_2$ is not too long and $(\epsilon,k)$ are chosen well enough, the path $\gamma_1$ cannot exit the cylinder since its length is about the same. Since $S^{n-1}$ is simply connected the map $G$ taking $Q_2$ to $Q_1$ is uniquely defined by this process and the choice of $P_1$ and $P_2$.
		The image of $S^{n-1}\times\{z_2\}$ under the map $G$ will be another constant mean curvature sphere as locally $G$ extends to an isometry from $g_2$ to $g_1$, this new constant mean curvature sphere will be nearly horizontal and pass through $P_1$, applying the inverse function theorem we know that such spheres are unique. This tells us that the imahge of $S^{n-1}\times\{z_2\}$ under $G$ is exactly the sphere $S^{n-1}\times\{z_1\}$, so that $\gamma_1$ stayed exactly horizontal.
		 It remains to check whether the orientations of the normal bundles in the cylinders to the two spheres agree in their images in $\mathcal{M}$. If they don't we can flip one of the cylinders and continue the argument. Then the spheres $S^{n-1}\times\{z_2+\mu\}$ will map to the spheres $S^{n-1}\times\{z_1+\mu\}$ under the obvious extension of $G$ using similar lifts, for $\mu$ near $0$ and hence for $\mu$ in some interval. This process lets us patch our cylinders together using $G$, which must be an isometry from $\bar{g_2}$ to $\bar{g_1}$ using the previous Lemma.
	\end{proof}

	\section{Maximal Normal Necks}	
	
	Lastly we will define a maximal neck and show that all our $(\epsilon,k)$-cylindrical geometric necks can be classified as either a maximal normal neck of finite length or that our manifold $\mathcal{M}$ is diffeomorphic to a quotient of $S^{n-1}\times\mathbb{R}$.

	\begin{definition}
		An $(\epsilon,k)$-cylindrical hypersurface neck $N$ is a maximal normal $(\epsilon,k)$-cylindrical hypersurface neck if $N$ is normal and if whenever $N^*$ is another such normal neck with $N=N^*F$ for some diffeomorphism $F$ then the map $F$ is onto.
	\end{definition}
	
We finish by showing a result from \cite{hamilton1997four} C2.5. We will show that we can classify our necks as finite maximal normal necks or $S^{n-1}\times S^1$.
	
	\begin{theorem}
		For any $\delta>0$ we can choose $\epsilon>0$ and $k$ so that any normal neck defined on a cylinder of length at least $3\delta$ which is geometrically $(\epsilon,k)$ cylindrical is contained in a maximal normal $(\epsilon,k)$ neck; or else the target manifold $M$ is diffeomorphic to a quotient of $S^{n-1}\times \mathbb{R}$ by a group of isometries in the standard metric.
	\end{theorem}
	
	\begin{proof}
		Since the neck $N$ has a domain cylinder of standard length at least $3\delta$, a point $P$ in the middle has standard distance at least $\delta$ from either end. If there is any other normal neck $N^*$ which is geometrically $(\epsilon,k)$ cylindrical with $N=N^*F$ for some $F$, then the previous theorem allows us to extend the definition of $N$ to a longer cylinder, and this extension $\bar{N}$ is unique, and now $N^*=\bar{N}\bar{F}$ for a map $\bar{F}$. Take the largest extension $\tilde{N}$ if $N$. It will be defined on $S^{n-1}\times B^1$ for some interval $B^1\subset \mathbb{R}$. If $B^1$ is of the form $[a,b]$ with $-\infty<a<b<\infty$ we have a maximal $(\epsilon,k)$ neck. If we have an interval $(a,b],(a,b]$ or $(a,b)$ with $-\infty<a<b<\infty$, we have enough bounds to extend the neck to the endpoints, so the original was not the largest. If $a=\infty$ but $b<\infty$ or vice-versa, then there must be two points $P_1$ and $P_2$ in the domain cylinder at different heights $z_1$ and $z_2$ with the same image in $\mathcal{M}$, because $\mathcal{M}$ has a finite volume and $N$ is clearly a local isometry so there must be considerable overlap. In fact we can make $P_1$ and $P_2$ at least $\delta$ from the finite end. Then the previous theorem shows that the neck $N$ must repeat itself, so both $a=\infty$ and $b=\infty$.
	\end{proof}
	
	\begin{remark}
		When we detect $S^{n-1}\times S^1$ we haven't glued together the cross sections $S^{n-1}\times\{a\}$ and $S^{n-1}\times\{b\}$, this is a more complicated case. What has happened is we have detected a return to the same cross section in $\mathcal{M}$, and due to uniqueness of these cross sections Lemma \ref{lemmaisom} no twisting/rotation can occur and we return with the same orientation. 
	\end{remark}
	
	\begin{remark}
		
			Given a cylinder $S^{n-1}\times[a,b]$ it is possible to glue the ends together $S^{n-1}\times[a,b]\large{/}\varphi$ where $\varphi$ is an orientation reversing homeomorphism $\varphi:S^{n-1}\times\{a\}\to S^{n-1}\times\{b\}$ such that this structure is topologically equivalent to $S^{n-1}\times S^1$. Regardless of the rotation of the cross sections at the ends $S^{n-1}\times\{a\}$ and $S^{n-1}\times\{b\}$.

			We can verify this as follows. We can think of this as a two-step process. We want choose an orientation of $S^{n-1}\times[a,b]$ such that we have an orientable manifold in the end. Suppose you want to glue $\{p\}\times\{a\}$ to $\{q\}\times\{b\}$. Then a small neighbourhood of $\{p\}\times\{a\}$ in $S^n\times\{a\}$ should be identified with a small neighbourhood of $\{q\}\times\{b\}$ in $S^{n-1}\times\{b\}$. These are two oriented discs, and you identify them by any orientation reversing homeomorphism. Then the resulting identification space is an oriented manifold with boundary. The boundary is a $(n-1)$-sphere. Hence glue to this an oriented ball, again identifying the boundary spheres by any orientation reversing homeomorphism. The result is homeomorphic with $S^{n-1}\times S^1$.
			
			What we have done is attached a $n$-dimensional $1$-handle to $S^{n-1}\times[a,b]$ with attaching region in different components of the boundary, and then attached a $n$-dimensional $n$-handle. We need only to make sure we attach with the right orientations. From this, you can define a homeomorphism with the standard $S^{n-1}\times S^1$.  The manifold will have a natural smooth structure at all points except at corner points, the union of which coincides with the boundary of the handle's base. This structure can be uniquely extended to a smooth structure on the entire manifold. Such extension is called smoothing of corners, refer to \cite{scorpan2005wild}.

			This can go wrong if we fail to choose the right orientation when attaching the $1$-handle. For example in dimension $3$ when we choose an orientation preserving homeomorphism, then a loop running along the $1$-handle and then connecting $\{p\}\times\{a\}$ and $\{q\}\times\{b\}$ in $S^2\times[a,b]$ would have the neighbourhood of a solid Klein bottle, not a torus.
		\end{remark}
	
	Now since the $\epsilon$ closeness is true even on the space-time region of the neck we are able to control the diffeomorphism type of the neck in a backward parabolic neighbourhood.  Moreover we can also control it in cases where surgery has occurred at an earlier time on a region adjacent to the neck. This is needed in the proof of Lemma 7.12 in \cite{huisken2009mean}, required to prove the neck continuation theorem, Theorem 8.2 \cite{huisken2009mean}.
	
	After we have completed the surgery process as described in section 3 \cite{huisken2009mean} we have attached a convex region diffeomorphic to the standard disk to a neck. This allows us to see that after each surgery the surgered region together with the long neck is it attached to is diffeomorphic to a standard disc.
	
	Lastly in the proof of the neck continuation theorem \cite{huisken2009mean}, Huisken and Sinestrari shows that in the case a neck does close, it does so to a standard convex cap diffeomorphic to a disc that is attached in the standard way to the standard neck. This shows that a neck type which ends in both directions will be diffeomorphic to the standard sphere $S^n$ because it consists of the standard cylinder glued to two standard discs without sphere twisting.

	\newpage
	\nocite{*}
	\bibliographystyle{plain}
	\bibliography{referencescross}

\end{document}